\documentclass[reqno,11pt]{amsart}
\usepackage{array,multirow}
\usepackage{amsmath}
\usepackage{amssymb}
\usepackage{amsfonts}
\usepackage{graphicx}
\usepackage{bm}
\usepackage{enumerate}
\usepackage{dsfont}
\usepackage[algo2e,linesnumbered,ruled]{algorithm2e}
\usepackage{float}
\usepackage{perpage}
\MakeSorted{figure}
\MakeSorted{table}

\usepackage[pagewise]{lineno} 
\usepackage{verbatim}
\usepackage{subfigure}
\usepackage{color}

\newtheorem{theorem}{Theorem}[section]
\newtheorem{proposition}[theorem]{Proposition}
\newtheorem{lemma}[theorem]{Lemma}
\newtheorem{corollary}[theorem]{Corollary}

\theoremstyle{definition}
\newtheorem{definition}[theorem]{Definition}

\newtheorem{algorithm}[theorem]{Algorithm}

\theoremstyle{remark}
\newtheorem{remark}[theorem]{Remark}

  \newcounter{mnote}
  \let\oldmarginpar\marginpar
    \renewcommand\marginpar[1]{\-\oldmarginpar[\raggedleft\footnotesize #1]%
    {\raggedright\footnotesize #1}}

\numberwithin{equation}{section}  
\numberwithin{figure}{section}  
\numberwithin{table}{section}  

\renewcommand\lll{|\kern-1pt|\kern-1pt|} 
 
\newcommand{\jump}[1]{\lbrack\!\lbrack #1 \rbrack\!\rbrack} 

\newcommand{\cT}{\mathcal{T}}

\newcommand{\Th}{\mathcal{T}_h}

\newcommand{\triplenorm}[1]{%
 \leftVert\kern-0.9pt\leftVert\kern-0.9pt\leftVert #1
 \rightVert\kern-0.9pt\rightVert\kern-0.9pt\rightVert}  

\newcommand{\RE}{\mathds{R}}

\begin{document}

\title[Multigrid for CR discretization of jump coefficients problems]
{Analysis of a multigrid preconditioner for Crouzeix-Raviart discretization of \\elliptic PDE 
       with jump coefficient}

\author[Y. Zhu]{Yunrong Zhu}
\email{zhu@math.ucsd.edu}
\address{Department of Mathematics, University of California at San Diego}

\date{\today}

\begin{abstract}
	In this paper, we present a multigrid $V$-cycle preconditioner for the linear system arising from piecewise linear nonconforming Crouzeix-Raviart  discretization of second order elliptic problems with jump coefficients. The preconditioner uses standard conforming subspaces as coarse spaces. We showed that the convergence rate of the multigrid $V$-cycle algorithm will deteriorate rapidly due to large jumps in coefficient. However, the preconditioned system has only a fixed number of small eigenvalues, which are deteriorated due to the large jump in coefficient, and the effective condition number is bounded logarithmically with respect to the mesh size. As a result, the multigrid $V$-cycle preconditioned conjugate gradient algorithm converges nearly uniformly. Numerical tests show both robustness with respect to jumps in the coefficient and the mesh size.
\end{abstract}

\maketitle


\section{Introduction}
\label{sec:intro}

In this paper, we present a multigrid preconditioner for the linear system arising from piecewise linear nonconforming Crouzeix-Raviart (CR) discretization of the following second order elliptic problems with jump coefficients: 
\begin{equation} \label{eqn:model}
	-\nabla \cdot (\kappa\nabla u) =f \mbox{  in  } \Omega, \qquad u|_{\partial \Omega} =0.
\end {equation}
Here, $\Omega\subset \RE^{d}$ ($d=2, 3$) is an open polygonal domain and $f\in L^2(\Omega)$. We assume that the diffusion coefficient $\kappa \in L^{\infty}(\Omega)$ is piecewise constant, namely, $\kappa(x)|_{\Omega_{m}} = \kappa_{m}$ is a constant for each (open) polygonal subdomain $\Omega_{m}$  satisfying $\cup_{m=1}^M
\overline{\Omega}_m=\overline{\Omega}$ and $\Omega_m
\cap \Omega_n =\emptyset$ for $m\neq n$.

Developing efficient multilevel solvers/preconditioners for the CR discretization is not only important for its own sake, but it has several important applications. In particular, it can be used to develop efficient solvers for other discretizations of \eqref{eqn:model}. For example, by using the equivalence between CR discretization and mixed methods (cf. \cite{Arnold.D;Brezzi.F1985}), the preconditioners for CR discretization can be applied to solve mixed finite element discretizations for \eqref{eqn:model} (see \cite{Brenner.S1992,Chen.Z1996} for the multigrid algorithms in the case of smooth coefficients). This relationship is also used in \cite{Hoppe.R;Wohlmuth.B1997} to analyze multigrid algorithms for mixed finite element discretization on adaptively refined meshes. Another important application is on the design of efficient multilevel solvers for interior penalty discontinuous Galerkin (IPDG) methods; we refer to \cite{Ayuso-de-Dios.B;Zikatanov.L2009} for the Poisson problem, and to \cite{Ayuso-de-Dios.B;Holst.M;Zhu.Y;Zikatanov.L2010} for \eqref{eqn:model} with jump coefficients. 

A lot of work has been done to develop multilevel solvers and preconditioners for the piecewise linear CR discretization of \eqref{eqn:model} in the context of smooth (or constant) coefficients. Multigrid algorithms were developed and analyzed in \cite{Brenner.S1989,Braess.D;Verfurth.R1990,Bramble.J;Pasciak.J;Xu.J1991,Vassilevski.P;Wang.J1995,Brenner.S2004b}, and two-level additive Schwarz preconditioners are presented in  \cite{Brenner.S1996}. Hierarchical and BPX-type multilevel preconditioners were proposed in \cite{Oswald.P1992e}, and their optimality was shown in \cite{Oswald.P2008}. Most of the aforementioned works define the multilevel structure using the natural sequences of nonconforming spaces. Since these nonconforming spaces are non-nested, special intergrid transfer operators (restriction and prolongation) are needed in the analysis and implementation of the algorithms.

On the other hand, we observe that the standard piecewise linear conforming finite element space is a proper subspace of the CR space on the same triangulation. Therefore,  one may be able to reduce the nonconforming discretization to the conforming one for which efficient and robust multilevel solvers/preconditioners are available. Due to the nested nature of the finite element spaces, one can also apply the standard multilevel analysis. Moreover, in the implementation of the algorithms one does not need special intergrid transfer operators; instead one could just use the natural inclusion operators. In this article, we propose a multigrid preconditioner for solving CR discretization, which consists of a standard smoother (Gauss-Seidel, or Jacobi iterative methods) on the fine nonconforming space and standard multigrid solver on the conforming subspaces.  The only difference between this algorithm and the standard multigrid algorithm on conforming spaces is the prolongation and restriction operators on the finest level. Since the spaces are nested, the prolongation is simply the natural inclusion operator from the conforming space to the CR space. Some preliminary numerical results of this paper has been reported in \cite{Ayuso-de-Dios.B;Holst.M;Zhu.Y;Zikatanov.L2011}. We remark that the idea of using conforming subspace as preconditioner for CR discretization in the context of smooth coefficients has been used in  \cite{Xu.J1989,Xu.J1996,Cowsar.L1993,Oswald.P1996}, and \cite{Sarkis.M1994} for jump coefficient case.

This article could be thought as an extension of the results in \cite{Ayuso-de-Dios.B;Holst.M;Zhu.Y;Zikatanov.L2010}, where a BPX-type (additive) multilevel preconditioner for CR discretization was developed, and was applied to solve the whole family of interior penalty DG discretizations (both symmetric and nonsymmetric) for \eqref{eqn:model}. No analysis on multiplicative multilevel preconditioner was given in \cite{Ayuso-de-Dios.B;Holst.M;Zhu.Y;Zikatanov.L2010}. In this article,  we discuss the robustness of a multiplicative preconditioner, the multigrid $V$-cycle preconditioner.  The analysis of the algorithm relies on some technical tools developed in \cite{Xu.J;Zhu.Y2008} and \cite{Ayuso-de-Dios.B;Holst.M;Zhu.Y;Zikatanov.L2010}. We show that the convergence rate of multigrid $V$-cycle algorithm will deteriorate rapidly due to large jump in the coefficient. On the other hand, we also show that the preconditioned system has only a fixed number of small eigenvalues deteriorated by the jump of the coefficients and mesh size, and the other eigenvalues are bounded nearly uniformly. Therefore, the resulting multigrid preconditioned PCG algorithm is robust with respect to the coefficient $\kappa$ and nearly optimal with respect to the mesh size $h$. 

Although the proof technique is similar to that in \cite{Xu.J;Zhu.Y2008} for the conforming case, we observe a major difference in the conclusion in Theorem~\ref{thm:vp} when $d=2$.  In the conforming case,  the condition number of $BA$ won't deteriorate due to the coefficients when $d=2$, as was pointed out in \cite[Remark 5.2]{Xu.J;Zhu.Y2008}. However, the condition number of $BA$  depends on the coefficient  in the current nonconforming case, as we can see from the numerical tests (cf. Table~\ref{tab:2d} and Figure~\ref{fig:eig2d}). 
 
The remaining of this paper is organized as follows. In Section \ref{sec:cr}, we give basic notation, the finite element discretizations. We also review the PCG convergence theory, and some properties of an interpolation operator introduced in \cite{Ayuso-de-Dios.B;Holst.M;Zhu.Y;Zikatanov.L2010}.  In Section \ref{sec:mg}, we present the multigrid algorithm, discuss its implementation, and prove the convergence rate and the robustness as a preconditioner. Finally, we give some numerical experiments to verify our theory in Section~\ref{sec:num}.

Throughout the paper, we will use the notation $x_1\lesssim y_1$, and $x_2\gtrsim y_2$,
whenever there exist constants $C_1, C_2$ independent of the mesh size
$h$ and the coefficient $\kappa$ or other parameters that $x_1$,
$x_2$, $y_1$ and $y_2$ may depend on, and such that $x_1 \le C_1 y_1$
and $x_2\ge C_2 y_2$.

\section{Preliminaries}
\label{sec:cr}
In this section, we introduce some basic notation and finite element spaces. We give an overview of the PCG convergence theory. We also recall some approximation and stability properties of the interpolation operator introduced in \cite{Ayuso-de-Dios.B;Holst.M;Zhu.Y;Zikatanov.L2010}, which play an essential role in the convergence analysis of the multigrid algorithm. 
\subsection*{Finite Element Spaces}
The weak form of the model problem \eqref{eqn:model} reads: Find $u\in H_{0}^{1}(\Omega)$ such that
\begin{equation} \label{eqn:weak}
a(u, v) :=( \kappa \nabla u, \nabla v)=(f, v) \quad \forall v\in H_{0}^{1}(\Omega)\,.
\end {equation}

We assume that there is an initial (quasi-uniform) triangulation 
$\mathcal{T}_{0}$ such that $\kappa_{T}:=\kappa(x)|_{T}$ is a constant for all $T\in\mathcal{T}_{0}$.
Let $\mathcal{T}_{l}:= \mathcal{T}_{h_{l}}$ ($l =1, \cdots, L$) be a family of uniform refinement  of $\mathcal{T}_{0}$ with mesh size $h_{l}$.  Without loss of generality, we assume that the mesh size $h_{l} \simeq 2^{L-l} h \;\;(l = 0, \cdots, L)$ for $h = h_{L}$. We denote $\mathcal{E}_{h}$ the set of all edges (in 2D) or faces (in 3D) of $\mathcal{T}_{h}$.
Let  $V^{CR}_{h}$ be the piecewise linear nonconforming Crouzeix-Raviart finite element space defined by: 
\begin{equation*}
V^{CR}_{h}\!=\!\left\{ v\in L^{2}(\Omega) \, : \, v_{|_{T}}
\in\, \mathbb{P}^{1}(T) \, \forall T \in \mathcal{T}_{h}\, \mbox{ and  }
\int_{e} 
\jump{v}_{e}ds=0 \,\, \forall\, e\in \mathcal{E}_{h}\right\},
\end{equation*}
where $\mathbb{P}^{1}(T)$ denotes the space of linear polynomials on $T$ and $\jump{v}_{e}$ denotes the jump across the edge/face $e\in \mathcal{E}_{h}$ with $\jump{v}_{e} = v$ when $e\subset \partial \Omega$.  
Then the $\mathbb{P}^{1}$-nonconforming finite element discretization of \eqref{eqn:model} reads:
\begin{equation}\label{eqn:cr}
  \mbox{Find } u\in V^{CR}_{h} :\;\; a_{h}(u, w) :=\displaystyle\sum_{T\in \mathcal{T}_{h}} \int _{T}\kappa_{T}\nabla u \cdot \nabla w=(f, w),  \quad\forall\, w\in V^{CR}_{h}. 
 \end{equation}
The bilinear form $a_{h}(\cdot, \cdot)$ induced a natural energy norm: $|v|_{h,\kappa}:=\sqrt{a_{h}(v,v)}$ for any $v\in V^{CR}_{h}$. We denote the full (broken) weighted $H^{1}$ norm by 
$\|v\|^{2}_{h,\kappa}:=\|v\|^{2}_{0,\kappa}+ |v|^{2}_{h,\kappa}.$ 
For any $w\in H^{1}(\Omega)$, let $\|w \|_{0,\kappa}$  and $|w|_{1,\kappa}$ be the weighted $L^{2}$ norm and weighted $H^{1}$ semi-norm defined as $$\|w\|^{2}_{0,\kappa}:=\int_{\Omega}\kappa |w|^{2} dx, \qquad  |w|_{1,\kappa}^{2} :=\int_{\Omega} \kappa |\nabla w|^{2} dx.$$ 
Let $A$ be the operator induced by the bilinear form $a_{h}(\cdot, \cdot)$, namely
$$
	(A v, w) =(v, w)_{A}: = a_{h}(v ,w),\qquad \forall v, w\in V_{h}^{CR}.
$$
In the whole paper, we will use the notation $\|\cdot \|_{A}$ or $|\cdot |_{h,\kappa}$ interchangeably for the energy norm.
Then solving \eqref{eqn:cr}  is equivalent to solve solve the linear system 
\begin{equation}
\label{eqn:linear}
	Au = f.
\end{equation}

In the multigrid preconditioner defined in the next section, on each level $l=0, \cdots, L$ we use standard conforming finite element discretization of \eqref{eqn:model} as coarse spaces: 
\begin{equation}
\label{prob:c}
\mbox{Find  } u_{l}\in V_{l}:\quad a(u_{l}, v_{l}) = (f, v_{l}), \qquad \forall v_{l} \in V_{l}.
\end{equation}
Here $V_{l} :=\{v\in H_{0}^{1}(\Omega): v|_{T} \in \mathbb{P}_{1}(T),\;\; \forall T\in \cT_{l}\}$ is the standard continuous piecewise $\mathbb{P}_{1}$-conforming finite element space defined on $\mathcal{T}_{l}$. For each $l = 0, \cdots, L$, we define the induced operator for \eqref{prob:c} as 
\[
(A_l v_l,w_l) = a(v_l,w_l), \quad \forall v_l,  w_l\in V_l.
\] 
In the sequel,  let us denote $V_{L+1}:= V_{h}^{CR}$ for simplicity.  We remark that all these finite element spaces are nested, that is,
\begin{equation}
\label{eqn:nested}
	V_{0} \subset\cdots\subset V_{L} \subset V_{L+1}.
\end{equation}
We also note that $a_{h}(v_{l}, w_{l}) = a(v_{l}, w_{l})$ when $v_{l}, w_{l} \in V_{l}$ for $l =0,1, \cdots, L$ since the space $V_{l}$ is conforming. Therefore, with a little abuse of notation, we will use $a_{h}(\cdot, \cdot)$ as the bilinear form in all the finite element spaces including the conforming ones. 

\subsection*{PCG Convergence }
Let $B$ be an symmetric positive definition (SPD) operator, known as a preconditioner. Instead of solving \eqref{eqn:linear} directly, we consider solving the following \emph{preconditioned system} by conjugate gradient method:
$$
	BA u = Bf.
$$
Since $A$ is SPD operator, $BA$ is also SPD with respect to the energy inner product $(\cdot, \cdot)_{A}$. Based on the standard PCG theory, the convergence rate of PCG algorithm can be bounded as 
$$
	\frac{\|u - u_{k}\|_{A}}{ \|u -u_{0}\|_{A}} \le 2\left( \frac{\sqrt{\mathcal{K}(BA)} -1}{ \sqrt{\mathcal{K}(BA)} +1}\right)^{k},
$$
where $u_{0}$ is the initial guess, $\{u_{k}\}$ is the solution sequence, and $\mathcal{K}(BA)$ is the (generalized) condition number of $BA$.  

This convergence rate is not sharp, in particular, when $BA$ contains some extreme eigenvalues. Suppose that the spectrum $\sigma(BA)$ of $BA$,
is divided in two sets:
$\sigma(BA)=\sigma_0(BA)\cup \sigma_1(BA)$, where
$\sigma_0(BA)=\{\lambda_1,\ldots ,\lambda_{m}\}$ contains of
all ``bad'' eigenvalues, and $\sigma_1(BA)=\{\lambda_{m+1}, \ldots , \lambda_{N}\}$ contains the
remaining eigenvalues, which are bounded below and above uniformly. Namely, there exist some constants $0<a<b$ such that $\lambda_j \in [a,b]$ for $j=m+1,\ldots, N$, with
$N=\mbox{dim}(V_h^{CR})$. Then,  the error
at the $k$-th iteration of PCG algorithm is bounded by 
(see e.g.~\cite{Axelsson.O1994,Hackbusch.W1994,Axelsson.O2003}):
\begin{equation}
\label{eqn:CG}
\|u-u_k\|_{A}\le 2(\mathcal{K}(BA)-1)^{m}
\left(\frac{\sqrt{b/a}-1}{\sqrt{b/a}+1}\right)^{k-{m}}\|u-u_0\|_{A}\;.
\end{equation}
This convergence rate estimate \eqref{eqn:CG} implies that if there are
only a few small eigenvalues of $BA$ in $\sigma_0(BA)$,
then the \emph{asymptotic} convergence rate of the resulting PCG
method will be dominated by the factor
$\frac{\sqrt{b/a}-1}{\sqrt{b/a}+1}$, i.e., by $\sqrt{b/a}$ where
$b=\lambda_{N}(BA)$ and $a=\lambda_{m+1}(BA)$.  We define $(b/a)$, which determines the asymptotic convergence rate, as \emph{effective condition number} (cf. \cite{Xu.J;Zhu.Y2008}):  
\begin{definition}
\label{def:effcond}
Let $V$ be a real $N$-dimensional Hilbert space, and
$T:V\to V$ be an SPD linear
operator, with eigenvalues $0 < \lambda_{1} \le \cdots \le
\lambda_{N}$. The \emph{$m$-th effective condition number} of
$T$ is defined by
    \begin{equation*}
	 \mathcal{K}_{m}(T) := \frac{\lambda_{N}(T)}{\lambda_{m+1}(T)}.
	 \end{equation*}
\end{definition}
Based on the estimate in \eqref{eqn:CG}, we need to analyze the spectrum of the preconditioned system carefully. That is what we are going to do for multigrid preconditioner. 
\subsection*{An Interpolation Operator}
Given any $l=0, 1, \cdots, L$,  let us review some main properties of an interpolation operator  $Q_{l}^{\kappa}: V_{h}^{CR} \to V_{l}$ introduced in \cite{Ayuso-de-Dios.B;Holst.M;Zhu.Y;Zikatanov.L2010}, which play a key role in the analysis of multilevel preconditioners. 
\begin{lemma}
\label{lm:interpolation}
There exists an interpolation operator $Q_{l}^{\kappa}: V_{h}^{CR} \to V_{l}$
satisfying the following approximation and stability properties:
\begin{align}
\mbox{ Approximation:  } \qquad \|(I-Q_{l}^{\kappa})v\|_{0,\kappa} & \lesssim  h_{l}|\log (2h_{l} /h)|^{\frac{1}{2}} \|v\|_{h,\kappa},\quad &\forall\, v\in V_{h}^{CR}, &&\label{eqn:app}\\
\mbox{ Stability:  } \qquad\qquad\qquad\qquad |Q_{l}^{\kappa}v|_{1,\kappa} & \lesssim |\log (2h_{l} /h)|^{\frac{1}{2}} \|v\|_{h,\kappa}\quad &\forall\, v\in V_{h}^{CR}, &&\label{eqn:stab}
\end{align}
with constants independent of the coefficient $\kappa$ and mesh size. 
\end{lemma}
A construction of such an operator  and the proof of the above results are given in the Appendix of \cite{Ayuso-de-Dios.B;Holst.M;Zhu.Y;Zikatanov.L2010}. We would like to
point out that the operators $Q_{l}^{\kappa}$ ($l=0, 1, \cdots, L$) are not used in the actual
implementation of the preconditioner. 

Observe that on the right hand side
of \eqref{eqn:app} and \eqref{eqn:stab}, the bounds are given in terms
of the (broken) weighted full $H^{1}$-norm $\|v\|_{h,\kappa}.$ In general,
one cannot replace the norm $\|v\|_{h,\kappa}$ by the energy norm $|v|_{h,\kappa}$ induced by the bilinear form $a_{h}(\cdot, \cdot)$. To replace this full weighted norm by the energy norm, we may use the Poincar\'e-Friedrichs inequality for the nonconforming finite element
space (cf. \cite{Dolejsi.V;Feistauer.M;Felcman.J1999,Brenner.S2003a}) to obtain:
\begin{eqnarray*}
	\|v\|^{2}_{0,\kappa} &\le& \left(\max_{T\in \Th} \kappa_{T}\right)  \int_{\Omega}|v|^{2} dx \lesssim  \left(\max_{T\in \Th} \kappa_{T}\right) \sum_{T\in \Th}\|\nabla v\|_{0,T}^{2} \lesssim \frac{\max_{T\in \Th} \kappa_{T}}{\min_{T\in \Th} \kappa_{T}} |v|^{2}_{h, \kappa}.
\end{eqnarray*}
By this inequality, we have 
\begin{corollary}
\label{cor:inter-jmp}
There exists an interpolation operator $Q_{l}^{\kappa}: V_{h}^{CR} \to V_{l}$
satisfying the following approximation and stability properties:
\begin{eqnarray}
		\|(I-Q_{l}^{\kappa}) v \|_{0,\kappa} & \lesssim& \mathcal{J}^{\frac{1}{2}}(\kappa) h_{l} |\log (2h_{l}/h)|^{\frac{1}{2}} |v |_{h,\kappa}\;, \quad \forall\, 	v\, \in V_{h}^{CR}\;, \label{eqn:japp}\\
	| Q_{l}^{\kappa} v  |_{1,\kappa} &\lesssim& \mathcal{J}^{{\frac{1}{2}}}(\kappa) |\log (2h_{l} /h)|^{\frac{1}{2}} |v |_{h,\kappa}\;, \quad \forall\, 	v\, \in V_{h}^{CR}\;, \label{eqn:jstab}
\end{eqnarray}	
where $\mathcal{J}(\kappa) = \max_{T\in \Th} \kappa_{T} / \min_{T\in \Th} \kappa_{T}$ is the jump of the coefficient.
\end{corollary}

Note the approximation and stability properties in Corollary~\ref{cor:inter-jmp} depend on the jump $\mathcal{J}(\kappa)$ of the coefficient. On the other hand, we may impose certain constraints on the finite element space $V_{h}^{CR}$ to get robust approximation and stability properties. For this purpose, let $\mathcal{I}$ be the index set of \emph{floating} subdomains (cf. \cite{Toselli.A;Widlund.O2005}), namely the subdomains not touching the Dirichlet boundary:
\begin{equation}\label{def:I}
 {\mathcal I}:=\left\{\,\, i \,\, :\,\, \mbox{meas}_{d-1} (\partial \Omega
\cap \partial \Omega_{i}) = 0\,\right\}\;.
\end{equation} 
We introduce the subspace
$\widetilde{V}_{h}^{CR} \subset V_{h}^{CR}$:
\begin{equation}\label{vtilde}
  \widetilde{V}_{h}^{CR}:= 
\left\{v \in V_{h}^{CR}: \int_{\Omega_{i}} v dx =0 ,\;\; \forall i\in
  {\mathcal I} \right\},
\end{equation}
The key feature of this subspace \eqref{vtilde} is that the
Poincar\'e-Friedrichs inequality for the nonconforming finite element
space (cf. \cite{Dolejsi.V;Feistauer.M;Felcman.J1999,Brenner.S2003a}) holds
on each subdomain, and this allows us to replace the full weighted $H^{1}$-norm
$\|v\|_{h,\kappa}$ by the energy norm $|v|_{h,\kappa}$, for any
$v\in\widetilde{V}_{h}^{CR}$. 

We remark that the condition $\int_{\Omega_{i}}v dx =0$ in \eqref{vtilde} is not essential; other conditions could be used (see \cite{Toselli.A;Widlund.O2005}) as
long as they allow for the application of a Poincar\'e-type
inequality. At this point, we would like to emphasize that the dimension of
$\widetilde{V}_{h}^{CR}$ is related to the number of floating
subdomains and in fact, $\dim(\widetilde{V}_{h}^{CR})={\rm dim}(V_{h}^{CR}) -
m_0$, where $m_{0} = \#\mathcal{I}$ is the cardinality of $\mathcal{I}$. 

By restricting now the action of the operator
$Q_{l}^{\kappa}$ to functions in $\widetilde{V}_{h}^{CR}$, we have the
following result, as an easy corollary from Lemma~\ref{lm:interpolation}. 
\begin{corollary}\label{cor:P}
  Let $\widetilde{V}_{h}^{CR}\subset V_h^{CR}$ be the subspace defined in
  \eqref{vtilde}. There exists an operator $Q_{l}^{\kappa} :
  V_{h}^{CR} \to V_{l}$ satisfying the following approximation and stability properties:
\begin{eqnarray}
		\|(I-Q_{l}^{\kappa}) v \|_{0,\kappa} &\lesssim& h_{l} |\log (2h_{l} /h)|^{\frac{1}{2}} |v |_{h,\kappa}\;, \quad \forall\, 	v\, \in \widetilde{V}_{h}^{CR},\label{eqn:wapp}\\
	| Q_{l}^{\kappa} v  |_{1,\kappa} &\lesssim& |\log (2h_{l}/h)|^{\frac{1}{2}} |v |_{h,\kappa}\;, \quad \forall\, 	v\, \in \widetilde{V}_{h}^{CR}.	\label{eqn:wstab}
\end{eqnarray}
\end{corollary}
\begin{proof}
	By the definition of $\widetilde{V}_{h}^{CR}$,  any $v\in \widetilde{V}_{h}^{CR}$ satisfies the Poincar\'e-Friedrichs inequality on each subdomain (cf. \cite{Dolejsi.V;Feistauer.M;Felcman.J1999,Brenner.S2003a}):
	$$
		\int_{\Omega_{i}} |v|^{2} dx \lesssim \sum_{T\in \cT_{h}, T\subset \Omega_{i}}\int_{T} |\nabla v|^{2} dx,\quad  i=1, \cdots, M.
	$$
	Since $\kappa$ is piecewise constant on each subdomain, we have 
	$$
		\|v\|_{0,\kappa} \lesssim |v|_{h,\kappa}, \quad \forall v\in \widetilde{V}_{h}^{CR}.
	$$
	The conclusion then follows by Lemma~\ref{lm:interpolation} and the above inequality.
\end{proof}
%

\section{A Multigrid Preconditioner}
\label{sec:mg}
The action of standard multigrid $V$-cycle iterative operator 
$B:=B_{L+1}:V_{L+1} \mapsto V_{L+1}$ on a given $g\in V_{L+1}$ can be defined 
recursively by the following algorithm (cf. \cite{Bramble.J1993}):
\begin{algorithm}[$V$-cycle]
\label{alg:vcycle} 
Let $g_{L+1}=g$, and $B_0=A_0^{-1}.$ For $l= 1, \cdots, L+1,$ we define
recursively $B_l g_l$ for any $g_l\in V_{l}$ by the following three steps:
\begin{enumerate}
  \item Pre-smoothing : $w_1=R_l g_l;$
  \item Subspace correction: $w_2=w_1+ B_{l-1}Q_{l-1}(g_l - A_l w_1);$
  \item Post-smoothing: $B_lg_l:=w_2+R_l^*(g_l-A_l w_2).$
\end{enumerate}
\end{algorithm}
Here, $T^{*}$ is the adjoint operator of $T$ with respect to the energy inner product $a_{h}(\cdot, \cdot)$, that is, 
$$a_{h}(T^{*} v, w) := a_{h}(v, Tw),\quad \forall v, w\in V_{h}^{CR}.$$
In this algorithm,  $Q_{l} : L^{2}(\Omega) \to V_{l}$ is the
standard $L^{2}$ projection defined by:
$$
	(Q_{l} v, w_{l} ) = (v, w_{l}), \qquad \forall w_{l} \in V_{l}, \;\; (l=0, \cdots, L).
$$
The operator $R_{l}:V_{l} \to V_{l}$  for $l=0,\cdots, L+1$ is called a smoother. Here we consider $R_{l}$ to be the Gauss-Seidel smoother, which can be interpreted as an successive subspace correction algorithm based on a one-dimensional subspace decomposition $ V_{l} = \sum_{i=1}^{n_{l}} V_{l}^{i}$ with $V_{l}^{i} = {\rm span}\{ \phi_{l}^{i}\}_{i=1}^{n_{l}}$ (cf. \cite{Xu.J1992a}). Here $n_{l}$ denotes the dimension of the space $V_{l}.$
\begin{algorithm}[Gauss-Seidel]
\label{alg:fgs} 
Let $w_{0} = 0 \in V_{l}$. Then for any $g_l\in V_{l}$, we define $R_{l} g_{l} :=w_{n_{l}}$ as the result of the following loop: for $i=1, \cdots, n_{l}$ do
\begin{enumerate}
  \item Find $e_{i} \in V_{l}^{i}$ : $a_{h}(e_{i}, \phi_{l}^{i})=(g_{l}, \phi_{l}^{i}) - a_{h}(w_{i-1}, \phi_{l}^{i});$
  \item Update: $w_{i}:=w_{i-1}+e_{i}.$
\end{enumerate}
\end{algorithm}
For $i=1,\cdots, n_l,$ let
$P_{l}^{i}: V_{L+1}\to V_{l}^{i}$ be the orthogonal projection defined by
\begin{eqnarray*}
	a_{h}(P_{l}^{i} u,\phi_{l}^{i}) &=& a_{h}(u,\phi_{l}^{i}), \;\; \forall \phi_{l}^{i}\in V_{l}^{i}.
\end{eqnarray*} 
Using this projection operator, the solution $e_{i}$ to the first step in Algorithm~\ref{alg:fgs} is 
$$
	e_{i} = P_{l}^{i} (A_{l}^{-1} g - w_{i-1}).
$$
Hence, the following recursive relation holds from Algorithm~\ref{alg:fgs}
$$
	A_{l}^{-1} g_{l} - w_{i} = (I-P_{l}^{i}) (A_{l}^{-1} g_{l} - w_{i-1}), \quad  i = 1, \cdots, n_{l}.
$$
Therefore, the Gauss-Seidel smoother $R_l: V_{l}\to
V_{l}$ defined in Algorithm~\ref{alg:fgs}  can be written as
\begin{equation}
\label{eqn:rl}
	R_l=\left(I-\prod_{i=1}^{n_l} (I-P_{l}^{i})\right)A_l^{-1}.
\end{equation}
On each level $l=0,1, \cdots, L+1$, we define the Galerkin projection $P_{l}: V_{L+1} \to V_{l}$ as 
$$a_{h}(P_{l} u,v_{l}) = a_{h}(u,v_{l}), \;\; \forall v_{l}\in V_{l}.$$

The implementation of Algorithm~\ref{alg:vcycle} is almost identical to the implementation of standard multigrid $V$-cycle (cf. \cite{Briggs.W;Henson.V;McCormick.S2000}). We use standard prolongation and restriction matrices for conforming finite element spaces  $V_{l}$ for $l=0, \cdots, L$. The corresponding matrices between $V_{L}$ and $V_{L+1}$, are however different. The prolongation matrix on $V_{L}$ can be viewed as the matrix representation of the natural inclusion $\mathcal{I}_{L}: V_{L} \to V_{L+1},$ which is defined by
$$(\mathcal{I}_{L} v)(x) = \sum_{e\in \mathcal{E}_{h}} v(m_{e}) \psi_{e}(x),$$
where $\psi_{e}$ is the CR basis on the edge/face $e\in \mathcal{E}_{h}$ and $m_{e}$ is the barycener of $e$. Therefore, the prolongation matrix has the same sparsity pattern as the edge-to-vertex (in 2D), or face-to-vertex (in 3D) connectivity, and each nonzero entry in this matrix equals the constant $1/d$ where $d$ is the space dimension. The restriction matrix is simply the transpose  of the prolongation matrix.

Now we are in position to discuss convergence of the multigrid $V$-cycle algorithm. The analysis is based on the following identity \cite{Xu.J;Zikatanov.L2002}.
\begin{lemma}\label{lm:xzid}
  Let $W$ be a Hilbert space, and $W_{j}\subset W$ for $j=0,\cdots, J$
  be a number of closed subspaces satisfying $W=\sum_{j=0}^{J} W_{j}.$
  Let $P_j:W\to W_{j}$ be the orthogonal projection with respect to the
  inner product of $W.$ Then
  $$\left\|(I-P_J)\cdots(I-P_0)\right\|^2_{{\mathcal
  L}(W)}=1-\frac{1}{1+c_0},$$ where
  $$c_0=\sup_{\left\|v\right\|=1}\inf_{v=\sum_{j=0}^{J} v_j}\sum_{j=0}^{J} \left\|P_j\left(\sum_{k=j+1}^{J} v_k\right)\right\|^2.$$
\end{lemma}

Recall that from \eqref{eqn:nested}, the finite element spaces $\{V_{l}\}_{l=0}^{L+1}$ are nested. Then we have the following space decomposition (recall that $V_{L+1} = V_{h}^{CR}$):
\begin{equation}\label{space:decom2}
	V_h^{CR}= \sum_{l=0}^{L+1} V_{l}.
\end{equation}
Combining with \eqref{eqn:rl} for the smoother $R_{l}$, the multigrid $V$-cycle operator $B:=B_{L+1}$ in Algorithm~\ref{alg:vcycle} has the representation (cf. \cite[(3.4)]{Bramble.J1993})
\begin{equation}
\label{eqn:errprop}
I-BA=\left((I-P_0)\prod_{l=1}^{L+1}\prod_{i=1}^{n_l}(I-P_{l}^{i})\right)^*
\left((I-P_0)\prod_{l=1}^{L+1}\prod_{i=1}^{n_l}(I-P_{l}^{i})\right),
\end{equation}
and the error propagation of the multigrid $V$-cycle iteration has the form:
$$
	u -u_{k} = (I-BA)^{k}(u -u_{0}).
$$
We define the convergence rate $\rho:=\|I-BA\|_{A}$, then obviously the above iteration is convergent if $\rho <1.$ From \eqref{eqn:errprop}, we have the following simple relationship between the extreme eigenvalues of $BA$ with the convergence rate $\rho$:
\begin{proposition}
	\label{prop:eig}
	We have the following estimate  for the maximum and minimum eigenvalues of $BA$:
	\begin{equation}\label{eq:mgemax}
\lambda_{\max}(BA)\le 1,\quad\mbox{ and } \quad \lambda_{\min}(BA)  = 1-\rho.
\end{equation}
\end{proposition}
\begin{proof}
Obviously, for any $v\in V_{h}^{CR}$ we have $$a_{h}((I-BA)v, v) = \left\|\left((I-P_0)\prod_{l=1}^{L+1}\prod_{i=1}^{n_l}(I-P_{l}^{i})\right)v\right\|_{A}^{2}\ge 0.$$ Therefore, $a_{h}(BAv,v)\le a_{h}(v,v), \;\; \forall v\in V_{h}^{CR}$, which implies that
$ \lambda_{\max}(BA)\le 1.$

By definition of $\|I-BA\|_{A}$, we have
$$
	\rho = \left\|I-B A\right\|_A = \sup_{0\neq v\in V_{h}^{CR}} \frac{a_{h}((I-BA)v,v)}{a_{h} (v,v)} = 1-  \inf_{0\neq v\in V_{h}^{CR}} \frac{a_{h}(BAv,v)}{a_{h} (v,v)}.
$$
From this identity, it is obvious that
\begin{equation*}
  \lambda_{\min}(BA)=\inf_{0\neq v\in V_{h}^{CR}}\frac{a_{h}(BA v, v)}{a_{h}(v,v)}= 1-\rho.
\end{equation*}
This completes the proof.
\end{proof}
On the other hand, notice that $I-BA$ is self-adjoint with respect to the  inner
product $a_{h}(\cdot,\cdot),$ we have
\begin{equation*}
\left\|I-B A\right\|_A= \sup_{\|v\|_{A} =1} a_{h}((I-BA)v, v) =\left\|(I-P_0)\prod_{l=1}^{L+1}\prod_{i=1}^{n_l}(I-P_{l}^{i})\right\|_A^2.
\end{equation*}
From Lemma \ref{lm:xzid}, we deduce that
\begin{equation}
\label{eqn:mgxz}
\rho =\left\|I-BA\right\|_A=\frac{c_0}{1+c_0},\; \mbox{ with }\;
c_0=\sup_{|v|_{h,\kappa}=1}\inf_{v=v_0+\sum_{l=1}^{L+1}\sum_{i=1}^{n_l}
v_{l}^{i}} c(v),
\end{equation}
where
$$c(v)=|P_0(v-v_0)|_{h,\kappa}^2+\sum_{l=1}^{L+1}\sum_{i=1}^{n_l}\left|P_{l}^{i}\left(
\sum_{(k,j)>(l,i)}v_{k}^{j}\right)\right|_{h,\kappa}^2,$$
with $v_{k}^{j} \in {\rm span}\{\phi_{k}^{j}\}$. 
In other words, we only need to estimate the upper bound of $c_0$ to obtain an upper bound of convergence rate $\rho$ of the multigrid $V$-cycle algorithm. By Proposition~\ref{prop:eig}, the upper bound of $\rho$ also provides a lower bound of minimum eigenvalue of $BA.$ 

For this purpose, we introduce the following lemma to bound $c(v).$
To simplify the notation, we set 
$Q_{L+1}^{\kappa}=I$, and $Q_{-1}^{\kappa}=0$. 
Then for any $v\in V_{L+1}$, we can decompose it as 
\begin{equation}
\label{eqn:vdecomp}
 v = \sum_{l=0}^{L+1} (Q_{l}^{\kappa} - Q_{l-1}^{\kappa}) v =\sum_{l=0}^{L+1} v_{l}, 
\quad\mbox{where} 
\quad v_l=(Q_{l}^{\kappa} - Q_{l-1}^{\kappa}) v.
\end{equation}
Clearly, $v_l\in V_l$ for $l =1, \cdots, (L+1)$ and $v_{0} = Q_{0}^{\kappa}
v\in V_0$. 

\begin{lemma}\label{lm:cv}
Given any $v\in V_{L+1}$ with the  decomposition \eqref{eqn:vdecomp}, we have
  \begin{equation}\label{eq:cv}
  c(v)\lesssim \sum_{l=0}^{L+1} \left|P_{l} v-Q_l^{\kappa} v\right|_{h,\kappa}^2+
\sum_{l=1}^{L+1}
h_l^{-2}\left\|(Q_l^{\kappa}-Q_{l-1}^{\kappa})v\right\|_{0,\kappa}^2.
\end{equation}
\end{lemma}
\begin{proof}
  By the decomposition \eqref{eqn:vdecomp}, we can write $v\in V_{L+1}$ as  
  $$v=\sum_{l=0}^{L+1} v_l= v_0+\sum_{l=1}^{L+1}
\sum_{i=1}^{n_l} v_{l}^{i}$$ with $v_l=(Q_l^{\kappa}-Q_{l-1}^{\kappa}) v.$
Noticing that 
$$\sum_{(k,j)>(l,i)}v_{k}^{j}:= \sum_{k=l}^{L+1} \sum_{j=i+1}^{n_{k}} v_{k}^{j}=(v-Q_l^{\kappa} v)+\sum_{j=i+1}^{n_l}
v_{l}^{j},$$ we have
\begin{eqnarray*}
  &&\sum_{l=1}^{L+1}\sum_{i=1}^{n_l}\left|P_{l}^{i}\sum_{(k,j)>(l,i)}v_{k}^{j}\right|_{h,\kappa}^2
  =\sum_{l=1}^{L+1}\sum_{i=1}^{n_l}\left|P_{l}^{i}\sum_{(k,j)>(l,i)}v_{k}^{j}\right|_{h,\kappa,\Omega_{l,i}}^2\\
  &&\le 2\sum_{l=1}^{L+1}\sum_{i=1}^{n_l}\left(\left|P_{l}^{i}(v-Q_l^{\kappa}
  v)\right|_{h,\kappa,\Omega_{l,i}}^2+\left|P_{l}^{i}\sum_{j=i+1}^{n_l}
v_{l}^{j}\right|_{h,\kappa,\Omega_{l,i}}^2\right)\\
&&\le2\sum_{l=1}^{L+1}\sum_{i=1}^{n_l}\left(\left|P_{l}^{i}(v-Q_l^{\kappa}
v)\right|_{h,\kappa,\Omega_{l,i}}^2+\left|\sum_{j=i+1}^{n_l}
v_{l}^{j}\right|_{h,\kappa,\Omega_{l,i}}^2\right),
\end{eqnarray*}
where $\Omega_{l,i}:={\rm supp}(\phi_{l}^{i})$ and $|w|_{h, \kappa, \Omega_{l,i}}^{2} := \sum_{T\in \cT_{l}, T\subset \Omega_{l,i}} \int_{T} \kappa_{T} |\nabla w|^{2} dx$.
By quasi-uniformity of the triangulations, it is not difficult to verify that
$$\sum_{i=1}^{n_l}\left|P_{l}^{i}(v-Q_l^{\kappa}
v)\right|_{h,\kappa,\Omega_{l,i}}^2\lesssim \left|P_{l} v-Q_l^{\kappa}
v\right|^2_{h,\kappa}$$ and by inverse inequality and quasi-uniformity of the triangulations
\begin{eqnarray*}
  \sum_{i=1}^{n_l}\left|\sum_{j=i+1}^{n_l}
  v_{l}^{j}\right|_{h,\kappa,\Omega_{l,i}}^2&\lesssim&
  \sum_{i=1}^{n_l}\sum_{j=i+1}^{n_l}
  h_l^{-2}\left\|v_{l}^{j}\right\|_{0,\kappa,\Omega_{l,i}}^2 \lesssim h_l^{-2}\left\|v_l\right\|_{0,\kappa}^2\\
  &=&  h_l^{-2}\left\|(Q_l^{\kappa}-Q_{l-1}^{\kappa})v\right\|_{0,\kappa}^2.
\end{eqnarray*}
Therefore, we get
$$c(v)\lesssim \sum_{l=0}^{L+1} \left|P_{l} v-Q_l^{\kappa} v\right|_{h,\kappa}^2+
\sum_{l=1}^{L+1} h_l^{-2}\left\|(Q_l^{\kappa}-Q_{l-1}^{\kappa})v\right\|_{0,\kappa}^2.$$
This completes the proof.
\end{proof}

Based on this lemma, in order to estimate the convergence rate of
the multigrid algorithm, we only need the stability and
approximation properties of the interpolation operators
$Q_l^{\kappa}$ ($l=0, \cdots, L$). Now we are in position to give an estimate on the convergence rate $\rho$ of the multigrid $V$-cycle Algorithm~\ref{alg:vcycle}.
\begin{theorem}
\label{thm:mgv}
	Let $B$ be the multigrid $V$-cycle iterator defined in Algorithm~\ref{alg:vcycle}. Then the convergence rate $\rho = \|I-BA\|_{A}$  satisfies
	\begin{equation}
	\label{eqn:rho}
		\rho \le 1- \frac{1}{1+ C_{0} \mathcal{J}(\kappa) L^{2}},
	\end{equation}
	where $C_{0}$ is a constant independent of the coefficient $\kappa$ and the mesh size $h$  (or number of levels  $L$).
	Moreover, the condition number $\mathcal{K}(BA)$ satisfies
	\begin{equation}
		\label{eqn:cond}
			\mathcal{K}(BA) \le 1+ C_{0} \mathcal{J}(\kappa) L^{2}.
	\end{equation}
\end{theorem}
\begin{proof}
 	Given any $v\in V_h^{CR},$ we make use of Corollary~\ref{cor:inter-jmp} to bound the right hand side of \eqref{eq:cv}.  For the first term, by the stability \eqref{eqn:jstab} of $Q_l^{\kappa}$  we have, for $l =0, 1, \cdots, L$
  \begin{eqnarray*}
    \left|P_{l} v-Q_l^{\kappa} v\right|_{h,\kappa}\le \left|v\right|_{h,\kappa}+\left|Q_l^{\kappa} v\right|_{h,\kappa}
    \lesssim \left(1+ \mathcal{J}^{\frac{1}{2}} (\kappa)\left|\log h_l\right|^{\frac{1}{2}}\right)\left|v\right|_{h,\kappa}.
  \end{eqnarray*}
 Notice that $Q_{L+1}^{\kappa} = I$, we obtain that
\begin{eqnarray*}
    \sum_{l=0}^{L+1}\left|P_{l} v-Q_l^{\kappa} v\right|_{h,\kappa}^2 &=&\sum_{l=0}^{L}\left|P_{l} v-Q_l^{\kappa} v\right|_{h,\kappa}^2 + \left|P_{L+1} v- v\right|_{h,\kappa}^2\\
    &\lesssim& \mathcal{J}(\kappa)\left(\sum_{l=0}^{L}\left|\log
    h_l\right|\right)\left|v\right|_{h,\kappa}^2\lesssim \mathcal{J}(\kappa)L^{2}\left|v\right|_{h,\kappa}^2.
\end{eqnarray*}
For the second term in the right hand side of \eqref{eq:cv}, by the approximation property \eqref{eqn:japp} of
$Q_{l-1}^{\kappa} \;(l=1,\cdots,L+1),$
\begin{eqnarray*}
  \left\|Q_l^{\kappa} v-Q_{l-1}^{\kappa} v\right\|_{0,\kappa}
    &\le&\left\|v-Q_{l}^{\kappa} v\right\|_{0,\kappa}+\left\|v-Q_{l-1}^{\kappa} v\right\|_{0,\kappa}\\
    &\lesssim& \mathcal{J}^{\frac{1}{2}}(\kappa)h_{l-1} \left|\log h_{l-1}\right|^{\frac{1}{2}}\left|v\right|_{h,\kappa}.
\end{eqnarray*}
Hence we have
$$\sum_{l=1}^{L+1}
h_l^{-2}\left\|(Q_l^{\kappa}-Q_{l-1}^{\kappa})v\right\|_{0,\kappa}^2\lesssim
\mathcal{J}(\kappa)\left(\sum_{l=1}^{L+1} |\log
h_{l-1}|\right)\left|v\right|_{h,\kappa}^2\lesssim \mathcal{J}(\kappa) L^2
\left|v\right|_{h,\kappa}^2.$$ 
Therefore, we obtain $c_0\le C_{0}\mathcal{J}(\kappa) L^2$ for some constant $C_{0}$ independent of coefficient $\kappa$ and meshsize $h$. The estimate \eqref{eqn:rho} then follows by the identity~\eqref{eqn:mgxz}.
Finally, the condition number estimate follows by \eqref{eq:mgemax} in Proposition~\ref{prop:eig} and \eqref{eqn:rho}.
\end{proof}

From Theorem~\ref{thm:mgv}, we conclude that the convergence rate of the multigrid $V$-cycle iteration is deteriorated by the jump of the coefficients rapidly. This phenomenon is justified also by the numerical experiment in Section~\ref{sec:num}. On the other hand, we can show that this deterioration due to the large jump in coefficient is only limited to a few fixed number of small eigenvalues in the preconditioned system. Therefore, the asymptotic convergence rate of the PCG algorithm is nearly uniform, as stated in the following theorem.
\begin{theorem}\label{thm:vp} 
Let $B$ be the multilevel preconditioner defined by Algorithm~\ref{alg:vcycle}. The effective condition number $\mathcal{K}_{m_{0}} (BA)$ satisfies 
\begin{equation}
\label{eqn:effcond}
	\mathcal{K}_{m_{0}}(BA) \le C_{1}^{2} L^{2},
\end{equation} 
where $m_{0} = \#\mathcal{I}$ is the number of floating subdomains (cf. \eqref{def:I}).  Therefore, the PCG algorithm has the following convergence rate estimate:
	\begin{equation*}
\frac{\|u-u_k\|_{A}}{\|u-u_{0}\|_{A}} \le 2(C_{0} \mathcal{J}(\kappa) L^{2})^{m_0}
\left(\frac{C_{1}L-1}{C_{1} L+1}\right)^{k-{m_0}}\;.
\end{equation*}
\end{theorem}
\begin{proof}
	To obtain the effective condition number, we restrict ourself in the subspace  $\tilde{V}_h^{CR} \subset V_{h}^{CR}$ defined by \eqref{vtilde}. By the identity~\eqref{eqn:mgxz}, we
have
$$\|I-BA\|_{A,{\tilde{V}_h^{CR}}}=\left\|(I-P^{0})\prod_{l=1}^{L+1}\prod_{i=1}^{n_l}(I-P_{l}^{i})\right\|_{A,\tilde{V}_h^{CR}}^2
=\frac{\tilde{c}_0}{1+\tilde{c}_0},$$ with $$\tilde{c}_0=\sup_{v\in
\tilde{V}_h^{CR}}\inf_{v=v_0+\sum_{l=1}^{L+1}\sum_{i=1}^{n_l}v_{l}^{i}}
c(v).$$ 
Given any $ v\in \tilde{V}_h^{CR},$ we make use of Corollary~\ref{cor:P} to bound the right hand side of \eqref{eq:cv}.  For the first term, by the stability \eqref{eqn:wstab} of $Q_l^{\kappa}$  we have, for $l =0, 1, \cdots, L$
  \begin{eqnarray*}
    \left|P_{l} v-Q_l^{\kappa} v\right|_{h,\kappa}\le \left|v\right|_{h,\kappa}+\left|Q_l^{\kappa} v\right|_{h,\kappa}
    \lesssim \left(1+ \left|\log h_l\right|^{\frac{1}{2}}\right)\left|v\right|_{h,\kappa}.
  \end{eqnarray*}
 Notice that $Q_{L+1}^{\kappa} = I$, we obtain that
\begin{eqnarray*}
    \sum_{l=0}^{L+1}\left|P_{l} v-Q_l^{\kappa} v\right|_{h,\kappa}^2 &=&\sum_{l=0}^{L}\left|P_{l} v-Q_l^{\kappa} v\right|_{h,\kappa}^2 + \left|P_{L+1} v- v\right|_{h,\kappa}^2\\
    &\lesssim&\left(\sum_{l=0}^{L}\left|\log
    h_l\right|\right)\left|v\right|_{h,\kappa}^2\lesssim L^{2}\left|v\right|_{h,\kappa}^2.
\end{eqnarray*}
For the second term in the right hand side of \eqref{eq:cv}, by the approximation property \eqref{eqn:wapp} of
$Q_{l-1}^{\kappa} \;(l=1,\cdots,L+1),$
\begin{eqnarray*}
  \left\|Q_l^{\kappa} v-Q_{l-1}^{\kappa} v\right\|_{0,\kappa}
    &\le&\left\|v-Q_{l}^{\kappa} v\right\|_{0,\kappa}+\left\|v-Q_{l-1}^{\kappa} v\right\|_{0,\kappa}\\
    &\lesssim& h_{l-1} \left|\log h_{l-1}\right|^{\frac{1}{2}}\left|v\right|_{h,\kappa}.
\end{eqnarray*}
Hence we have
$$\sum_{l=1}^{L+1}
h_l^{-2}\left\|(Q_l^{\kappa}-Q_{l-1}^{\kappa})v\right\|_{0,\kappa}^2\lesssim
\left(\sum_{l=1}^{L+1} |\log
h_{l-1}|\right)\left|v\right|_{h,\kappa}^2\lesssim L^2
\left|v\right|_{h,\kappa}^2.$$ 
Therefore, we obtain that  $\tilde{c}_0\lesssim L^2.$ Then by mini-max Theorem (cf. \cite[Theorem 8.1.2]{Golub.G;Van-Loan.C1996}) we obtain
\begin{eqnarray*}
  \lambda_{m_0+1}(BA)&\ge&\min_{0\neq v\in
    \tilde{V}_h^{CR}}\frac{a(BAv,v)}{a(v,v)} =  \frac{1}{1+\tilde{c}_0}
    \gtrsim L^{-2}.
\end{eqnarray*}
Together with \eqref{eq:mgemax},  we get the estimate on the effective condition number of $BA$:
$$\mathcal{K}_{m_0}(BA)\le C_{1}^{2} L^2,$$
for some constant $C_{1}$ independent of the coefficient $\kappa$ and mesh size $h$ (or levels $L$). The convergence rate estimate of the PCG algorithm follows by \eqref{eqn:cond}, \eqref{eqn:effcond} and \eqref{eqn:CG}. 
This completes the proof.
\end{proof}
\begin{remark}
	We remark that $m_{0} = \#\mathcal{I}$ gives an upper bound of the total number small eigenvalue. In some cases, the number of bad eigenvalues might be less than $m_{0}$, as one can observe from the numerical tests. Since $m_{0}$ is a fixed integer depending only on the distribution of the coefficient, the asymptotic convergence rate of the PCG algorithm is dominated by $\frac{C_{1}L-1}{C_{1} L+1}$, which is independent of the coefficient.
\end{remark}
%

\section{Numerical Results}
\label{sec:num}
In this section, we present some numerical experiments in both 2D and 3D domains to justify the performance of the multigrid $V$-cycle preconditioner in Algorithm \ref{alg:vcycle}. 

In the first example, we consider the model problem \eqref{eqn:model} in the square $\Omega = [-1,1]^{2}$ with coefficients $\kappa =1$ in the subdomain $[-0.5,1]^{2}\cup [0, 0.5]^{2}$, and $\kappa = \varepsilon$ for the remaining subdomain, where $\varepsilon$ varies (see Figure~\ref{fig:domain2d}). We start with a structured initial triangulation on level 0 with mesh size $h = 2^{-1}$ which resolves the jump interface of the coefficient. Then on each level, we uniformly refine the mesh by subdividing each element into four congruent ones. In this example, we use 1 forward/backward Gauss-Seidel iteration as pre/post smoother in the multigrid preconditioner, and the stopping criteria of the PCG algorithm is $\|r_{k}\| / \|r_{0}\| <10^{-7}$ where $r_{k}= f-Au_{k}$ is the the residual at $k$-th iteration.

Figure~\ref{fig:eig2d} shows the eigenvalue distribution of the multigrid $V$-cycle preconditioned system $BA$ when $h=2^{-5}$ and $\varepsilon = 10^{-5}.$ As we can see from this figure, there is only one small eigenvalue deteriorated by the coefficient and mesh size.  This is different from the conforming case. As we know in the conforming case (cf.  \cite[Remark 5.2]{Xu.J;Zhu.Y2008}), there will be no such extremely small eigenvalue. However, we still observe a small eigenvalue in Figure~\ref{fig:eig2d} in this nonconforming case.
\begin{figure}[htbp]
\centering
	\parbox{0.45\textwidth}{
       \includegraphics[width=0.45\textwidth]{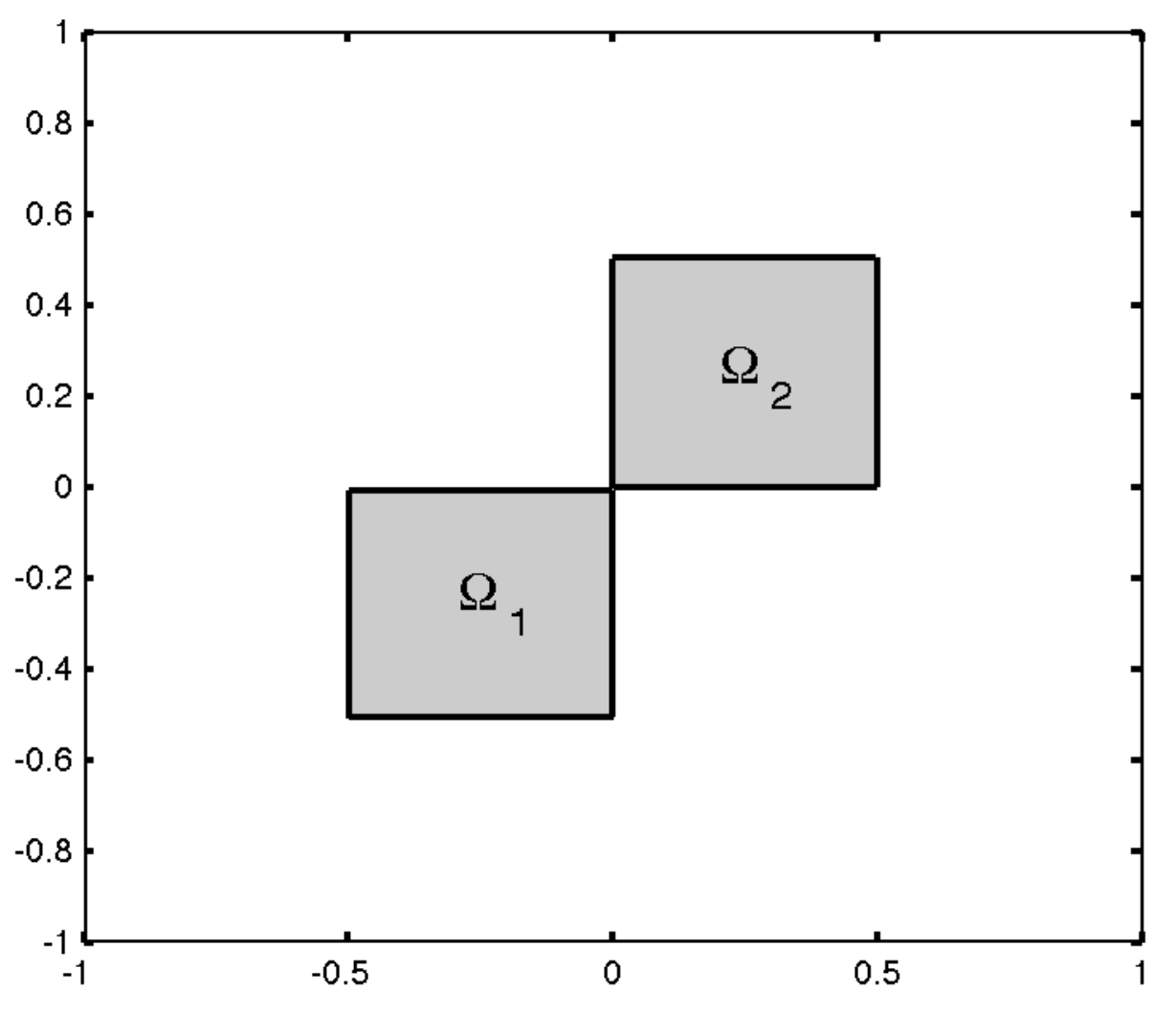}
       \caption{2D Computational Domain}
       \label{fig:domain2d}}
       \quad
       \begin{minipage}{0.45\textwidth}
       \includegraphics[width=0.99\textwidth]{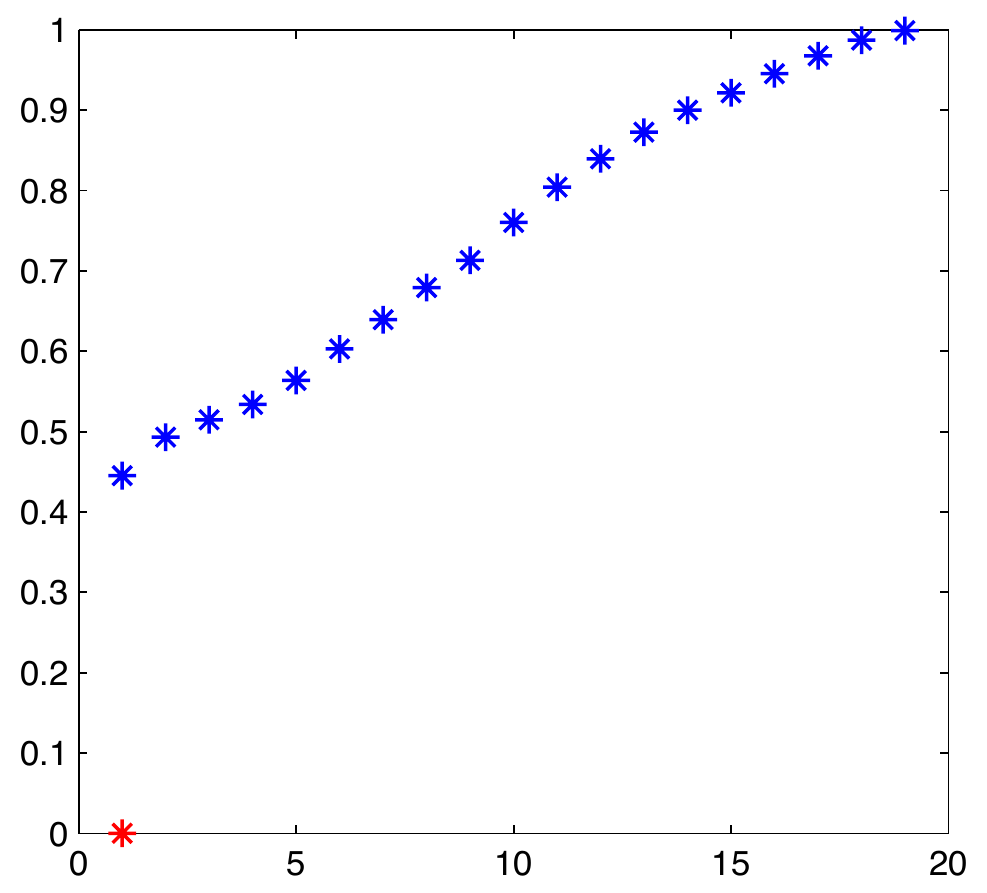}
       \caption{Eigenvalue Distribution of $BA$ in 2D}
       \label{fig:eig2d}
       \end{minipage}
\end{figure}

Table \ref{tab:2d} shows the estimated condition number $\mathcal{K}$ (number of PCG iterations) and the effective condition number $\mathcal{K}_{1}$ of the multigrid preconditioned system on $V_{h}^{CR}$. As we can see from this table, although the condition number is deteriorated due to the jumps in the coefficient and mesh size, the corresponding effective condition number is nearly uniform, which coincides with the theory. 
\begin{table}[htbp]
	\begin{tabular}{c|c||c|c|c|c|c}
\hline
{$\varepsilon$} & levels &  0         & 1        & 2        & 3        & 4 \\
\hline\hline
\multirow{2}{*}{$1$}
 & $\mathcal{K}$ &  1.65 (8)&  1.83 (10)&   1.9 (10)&   1.9 (10)&  1.89 (10)\\
 & $\mathcal{K}_{1}$ &  1.44 &  1.78 &  1.77 &  1.78 &  1.76 \\\hline
\multirow{2}{*}{$10^{-1}$}
 & $\mathcal{K}$ &  3.78 (10)&  3.69 (11)&  3.76 (12)&  3.79 (12)&  3.88 (12)\\
 & $\mathcal{K}_{1}$ &  1.89 &  1.87 &  1.93 &  1.92 &  1.95 \\\hline
\multirow{2}{*}{$10^{-2}$}
 & $\mathcal{K}$ &  23.4 (12)&  23.6 (13)&  24.6 (13)&  25.1 (14)&    26 (15)\\
 & $\mathcal{K}_{1}$ &  2.15 &  1.96 &  1.99 &  1.97 &  2.24 \\\hline
\multirow{2}{*}{$10^{-3}$}
 & $\mathcal{K}$ &   218 (13)&   223 (14)&   232 (15)&   238 (16)&   246 (16)\\
 & $\mathcal{K}_{1}$ &  2.19 &  1.98 &     2 &  1.98 &  2.29 \\\hline
\multirow{2}{*}{$10^{-4}$}
 & $\mathcal{K}$ & 2.17e+03 (14)& 2.21e+03 (15)& 2.31e+03 (16)& 2.37e+03 (18)& 2.45e+03 (18)\\
 & $\mathcal{K}_{1}$ &   2.2 &  1.98 &     2 &  1.98 &   2.3 \\\hline 
\multirow{2}{*}{$10^{-5}$}
 & $\mathcal{K}$ & 2.17e+04 (15)& 2.21e+04 (16)& 2.31e+04 (17)& 2.37e+04 (20)& 2.76e+04 (21)\\
 & $\mathcal{K}_{1}$ &   2.2 &  1.98 &     2 &  1.98 &  2.64 \\\hline
\end{tabular}
\caption{Estimated condition number $\mathcal{K}$ and the effective condition number $\mathcal{K}_{1}$ in 2D}
\label{tab:2d}
\end{table}

As a second example, we consider the model problem in a 3D unit cube with similar setting. We set the coefficient $\kappa =1.0$ in subdomains $\Omega_{1} =[0.25, 0.5]^{3}$ and $\Omega_{2}=[0.5, 0.75]^{3};$ and $\kappa =\varepsilon$ for the remaining subdomain (see Figure~\ref{fig:domain3d}). The initial mesh (level =0) has a mesh size $h_{0} =2^{-2}$, which resolves the jump interface. In this example, we used 5 forward/backward Gauss-Seidel as smoother in the multigrid preconditioner, and  the stopping criteria $\|r_{k}\| / \|r_{0}\| <10^{-12}$ for the PCG algorithm.

\begin{figure}[htbp]
\centering
	\parbox{0.45\textwidth}{
       \includegraphics[width=0.39\textwidth]{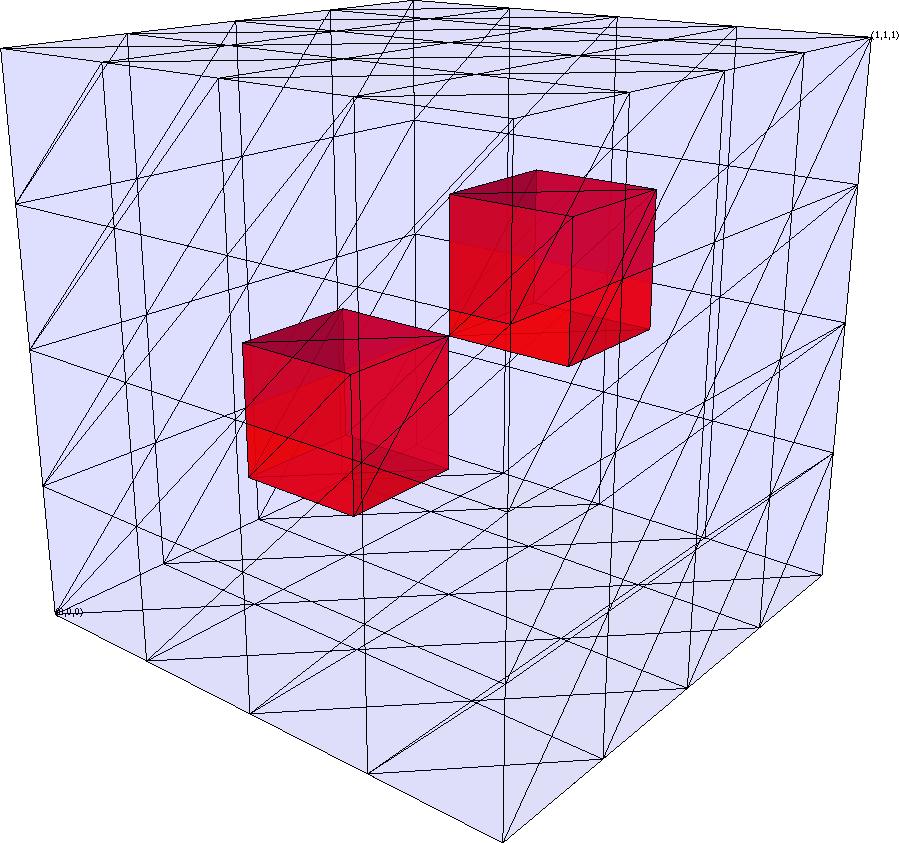}
       \caption{3D Computational Domain}
       \label{fig:domain3d}}
       \quad
       \begin{minipage}{0.45\textwidth}
       \includegraphics[width=0.99\textwidth]{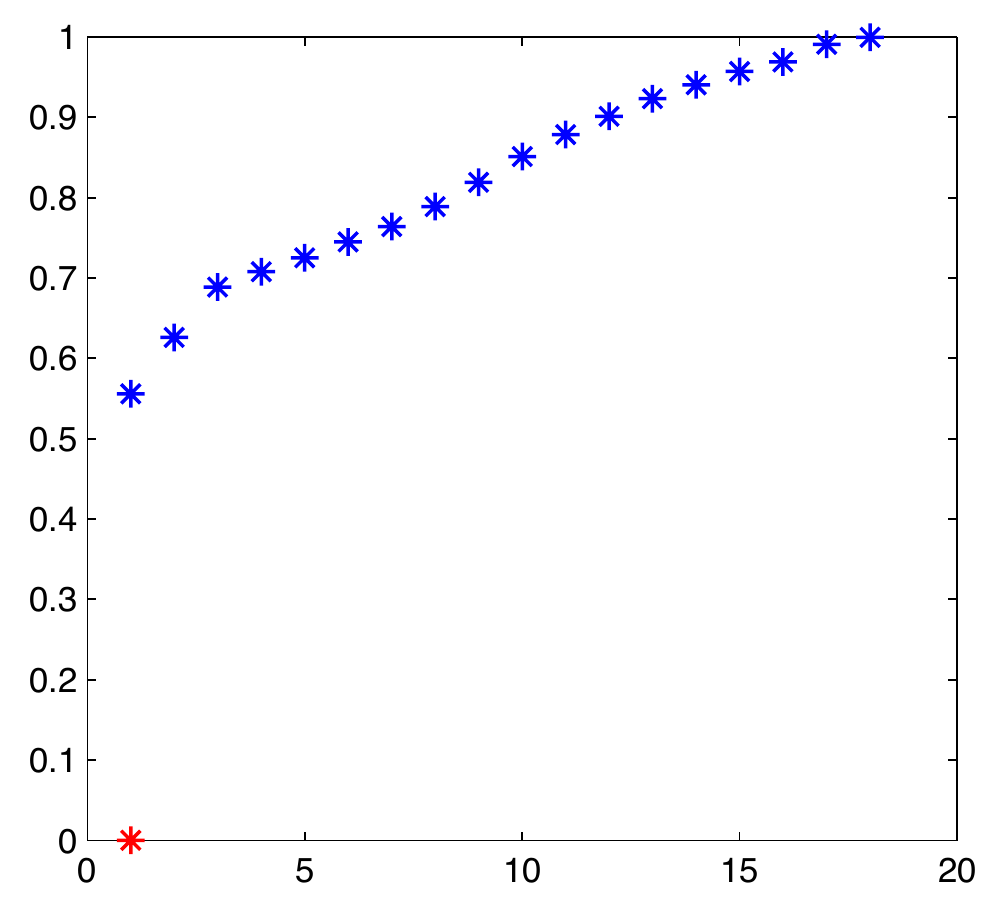}
       \caption{Eigenvalue Distribution of $BA$ in 3D}
       \label{fig:eig3d}
       \end{minipage}
\end{figure}
Figure~\ref{fig:eig3d} shows the eigenvalue distribution of the multigrid $V$-cycle preconditioned system $BA$ when $h=2^{-5}$ and $\varepsilon = 10^{-5}.$ As before, this figure shows that there is only one small eigenvalue deteriorated by the coefficient and mesh size. 

\begin{table}
\centering
\begin{tabular}{c|c||c|c|c|c}
\hline
 $\varepsilon$ & levels &  0         & 1        & 2        & 3        \\
\hline\hline
\multirow{2}{*}{$1$}
 & $\mathcal{K}$ & 1.19 (8)& 1.34 (11)& 1.37 (11)& 1.36 (11)\\
 &$\mathcal{K}_{1}$ & 1.16 &  1.26 &  1.31& 1.29\\\hline
\multirow{2}{*}{$10^{-1}$}
 & $\mathcal{K}$ & 2.3 (10)& 1.94(13)& 1.75 (13)& 1.67 (14)\\
 & $\mathcal{K}_{1}$ & 1.60 & 1.56 & 1.45 &  1.43 \\\hline
\multirow{2}{*}{$10^{-3}$}
 & $\mathcal{K}$ &   86.01 (11)&   63.07 (16)&  52.67 (17)& 48.19(17)\\
 & $\mathcal{K}_{1}$ & 2.4 &  2.12  &   1.89 &  1.78 \\\hline
\multirow{2}{*}{$10^{-5}$}
 & $\mathcal{K}$ &  8.39+03 (13)&  6.15e+03 (18)&  5.13e+03 (19)&   4.70e+03(19)\\
 &  $\mathcal{K}_{1}$ &   2.44 &   2.14 &  1.91 & 1.80 \\\hline
\multirow{2}{*}{$10^{-7}$}
 & $\mathcal{K}$ &  8.39+05 (14)&  6.15e+05 (21)&  5.13e+05 (23)&  4.70e+05(21)\\
 & $\mathcal{K}_{1}$ &  2.45 &  2.14  & 1.91  & 1.80\\\hline
\end{tabular}
\caption{Estimated condition number $\mathcal{K}$ (PCG iterations) and effective condition number $\mathcal{K}_{1}$ for multigrid $V$-cycle preconditioner in 3D}
{\label{tab:3d}}
\end{table}
Table \ref{tab:3d} shows the estimated condition number $\mathcal{K}$ (with the number of PCG iterations), and the effective condition number $\mathcal{K}_{1}$. As we can see from this table, the condition number $\mathcal{K}$ gets bigger as the jump $\mathcal{J}(\kappa) = 1/\varepsilon$ becomes larger. Thus the condition number deteriorates due to the jump of coefficient. However, the effective condition number $\mathcal{K}_{1}$ remains in a small range. This justifies that the PCG algorithm with multigrid $V$-cycle preconditioner is a robust solver for solving \eqref{eqn:model} with jump coefficients.

Table \ref{tab:2mg} and \ref{tab:5mg} shows the estimated convergence rate of the multigrid $V$-cycle with 2 Gauss-Seidel smoothers and 5 Gauss-Seidel smoothers respectively, with respect to different coefficients $\varepsilon=10^{-i},\; i=0, 1, \cdots, 5$ (rows) and different levels $L=0,1, \cdots ,3$ (columns).  The convergence rate of the multigrid $V$-cycle deteriorates rapidly with respect to the jump. Actually, when $\varepsilon\le 10^{-2}$ the multigrid $V$-cycle is unfavorable.  When the jump is not so severe, the more smoothing step, the faster the algorithm converges. However, for problems with large jump, more smoothing steps does not improve the convergence rate much. 

Another interesting phenomenon we observe from this two tables is that the convergence rate seems to be uniform with respect to the number of levels, when the jump is fixed. The same phenomenon can also be observed from Table~\ref{tab:3d} for the effective condition numbers. This is also different from the conforming case in \cite[Table 2]{Xu.J;Zhu.Y2008}, where we could observe some growth of the convergence rate with respect to the number of levels. The theoretical investigation of this phenomenon will be a future work. 
\begin{table}[htbp]
\begin{center}
\begin{tabular}{c|c||c|c|c|c|c|c}\hline
    &     & \multicolumn{6}{c}{$\varepsilon$}\\
    \hline
     Levels & h  & 1 & $10^{-1}$ & $10^{-2}$ & $10^{-3}$ &$10^{-4}$ &$10^{-5}$ \\
    \hline \hline

     0& $2^{-2}$ &  0.397 &  0.731 & 0.930 & 0.991  & 0.999 & 0.9999\\
     1& $2^{-3}$ & 0.501 & 0.675 & 0.921 & 0.990  & 0.999 & 0.9999\\
     2& $2^{-4}$ & 0.517 & 0.647 &  0.905 & 0.988  &  0.999 & 0.9999\\
     3& $2^{-5}$ & 0.524 & 0.634 & 0.895 & 0.987  & 0.999 & 0.9999\\
    \hline
\end{tabular}
\caption{Multigrid $V$-cycle convergence rate (2 Gauss-Seidel Smoothers) v.s. jump and level}{\label{tab:2mg}}
\end{center}
\end{table}

\begin{table}[htbp]
\begin{center}
\begin{tabular}{c|c||c|c|c|c|c|c}\hline
    &     & \multicolumn{6}{c}{$\varepsilon$}\\
    \hline
     Levels & h  & 1 & $10^{-1}$ & $10^{-2}$ & $10^{-3}$ &$10^{-4}$ &$10^{-5}$ \\
    \hline \hline

     0& $2^{-2}$ &  0.152 &  0.575 & 0.904 & 0.988  & 0.999 & 0.9999\\
     1& $2^{-3}$ & 0.254 & 0.485 & 0.870 & 0.984  & 0.998 & 0.9998\\
     2& $2^{-4}$ & 0.269 & 0.429 &  0.846 & 0.981  &  0.998 & 0.9998\\
     3& $2^{-5}$ & 0.286 & 0.403 & 0.832 & 0.979  & 0.998 & 0.9998\\
    \hline
\end{tabular}
\caption{Multigrid $V$-cycle convergence rate (5 Gauss-Seidel Smoothers) v.s. jump and level}{\label{tab:5mg}}
\end{center}
\end{table}

\subsection*{Acknowledgements}
This work was supported in part by NSF DMS-0715146 and DTRA Award HDTRA-09-1-0036. The author would also like to express his gratitude to Blanca Ayuso de Dios, Michael Holst and Ludmil Zikatanov for their invaluable support and encouragement on this work.



\begin{thebibliography}{10}

\bibitem{Arnold.D;Brezzi.F1985}
D.~N. Arnold and F.~Brezzi.
\newblock Mixed and nonconforming finite element methods: Implementation,
  postporcessing and error estimates.
\newblock {\em RAIRO Model Math. Anal. Numer.}, 19:7--32, 1985.

\bibitem{Axelsson.O1994}
O.~Axelsson.
\newblock {\em Iterative solution methods}.
\newblock Cambridge University Press, Cambridge, 1994.

\bibitem{Axelsson.O2003}
O.~Axelsson.
\newblock Iteration number for the conjugate gradient method.
\newblock {\em Mathematics and Computers in Simulation}, 61(3-6):421--435,
  2003.
\newblock MODELLING 2001 (Pilsen).

\bibitem{Ayuso-de-Dios.B;Holst.M;Zhu.Y;Zikatanov.L2010}
B.~Ayuso~de Dios, M.~Holst, Y.~Zhu, and L.~Zikatanov.
\newblock {Multilevel Preconditioners for Discontinuous Galerkin Approximations
  of Elliptic Problems with Jump Coefficients}.
\newblock {\em Arxiv preprint arXiv:1012.1287}, 2010.

\bibitem{Ayuso-de-Dios.B;Holst.M;Zhu.Y;Zikatanov.L2011}
B.~Ayuso~de Dios, M.~Holst, Y.~Zhu, and L.~Zikatanov.
\newblock {Multigrid Preconditioner for Nonconforming Discretization of
  Elliptic Problems with Jump Coefficients}.
\newblock {\em Submitted to DD20 Prodeedings, available at arXiv:1107.2160v1},
  2011.

\bibitem{Ayuso-de-Dios.B;Zikatanov.L2009}
B.~Ayuso~de Dios and L.~Zikatanov.
\newblock {Uniformly Convergent Iterative Methods for Discontinuous Galerkin
  Discretizations}.
\newblock {\em Journal of Scientific Computing}, 40(1):4--36, 2009.

\bibitem{Braess.D;Verfurth.R1990}
D.~Braess and R.~Verf{\"u}rth.
\newblock Multigrid methods for nonconforming finite element methods.
\newblock {\em SIAM Journal on Numerical Analysis}, 27:979--986, 1990.

\bibitem{Bramble.J1993}
J.~H. Bramble.
\newblock {\em Multigrid Methods}, volume 294 of {\em Pitman Research Notes in
  Mathematical Sciences}.
\newblock Longman Scientific \& Technical, Essex, England, 1993.

\bibitem{Bramble.J;Pasciak.J;Xu.J1991}
J.~H. Bramble, J.~E. Pasciak, and J.~Xu.
\newblock The analysis of multigrid algorithms with nonnested spaces or
  noninherited quadratic forms.
\newblock {\em Mathematics of Computation}, 56:1--34, 1991.

\bibitem{Brenner.S1989}
S.~C. Brenner.
\newblock An optimal order multigrid for {P}1 nonconforming finite elements.
\newblock {\em Mathematics of Computation}, 52:1--15, 1989.

\bibitem{Brenner.S1992}
S.~C. Brenner.
\newblock A multigrid algorithm for the lowest-order {R}aviart-{T}homas mixed
  triangular finite element method.
\newblock {\em SIAM Journal on Numerical Analysis}, 29:647--678, 1992.

\bibitem{Brenner.S1996}
S.~C. Brenner.
\newblock Two-level additive {S}chwarz preconditioners for nonconforming finite
  element methods.
\newblock {\em Mathematics of Computation}, 65:897--921, 1996.

\bibitem{Brenner.S2003a}
S.~C. Brenner.
\newblock {P}oincar\'e--{F}riedrichs inequalities for piecewise {$H^1$}
  functions.
\newblock {\em SIAM Journal on Numerical Analysis}, 41(1):306--324, 2003.

\bibitem{Brenner.S2004b}
S.~C. Brenner.
\newblock Convergence of nonconforming {V}-cycle and {F}-cycle multigrid
  algorithms for second order elliptic boundary value problems.
\newblock {\em Math. Comp.}, 73(247):1041--1066 (electronic), 2004.

\bibitem{Briggs.W;Henson.V;McCormick.S2000}
W.~L. Briggs, V.~E. Henson, and S.~F. McCormick.
\newblock {\em A multigrid tutorial}.
\newblock Society for Industrial and Applied Mathematics (SIAM), Philadelphia,
  PA, second edition, 2000.

\bibitem{Chen.Z1996}
Z.~Chen.
\newblock Equivalence between and multigrid algorithms for nonconforming and
  mixed methods for second order elliptic problems.
\newblock {\em East-West Journal of Numerical Mathematics}, 4:1--33, 1996.

\bibitem{Cowsar.L1993}
L.~Cowsar.
\newblock Domain decomposition methods for nonconforming finite element spaces
  of lagrange-type.
\newblock In {\em The Sixth Copper Mountain Conference on Multigrid Methods},
  pages 93--109, 1993.

\bibitem{Dolejsi.V;Feistauer.M;Felcman.J1999}
V.~Dolej{\v{s}}{\'{\i}}, M.~Feistauer, and J.~Felcman.
\newblock On the discrete {F}riedrichs inequality for nonconforming finite
  elements.
\newblock {\em Numer. Funct. Anal. Optim.}, 20(5-6):437--447, 1999.

\bibitem{Golub.G;Van-Loan.C1996}
G.~H. Golub and C.~F. Van~Loan.
\newblock {\em Matrix computations}.
\newblock Johns Hopkins Studies in the Mathematical Sciences. Johns Hopkins
  University Press, Baltimore, MD, third edition, 1996.

\bibitem{Hackbusch.W1994}
W.~Hackbusch.
\newblock {\em {Iterative Solution of Large Sparse Systems of Equations}},
  volume~95 of {\em Applied Mathematical Sciences}.
\newblock Springer-Verlag New York, Inc., 1994.

\bibitem{Hoppe.R;Wohlmuth.B1997}
R.~H.~W. Hoppe and B.~Wohlmuth.
\newblock Adaptive multilevel techniques for mixed finite element
  discretizations of elliptic boundary value problems.
\newblock {\em SIAM Journal on Numerical Analysis}, 34(4):1658--1681, aug 1997.

\bibitem{Oswald.P1992e}
P.~Oswald.
\newblock On hierarchical basis multilevel method with nonconforming {P}1
  elements.
\newblock {\em Numerische Mathematik}, 62:189--212, 1992.

\bibitem{Oswald.P1996}
P.~Oswald.
\newblock Preconditioners for nonconforming discretizations.
\newblock {\em Mathematics of Computation}, 65(215):923--941, 1996.

\bibitem{Oswald.P2008}
P.~Oswald.
\newblock {Optimality of multilevel preconditioning for nonconforming P1 finite
  elements}.
\newblock {\em Numerische Mathematik}, 111(2):267--291, 2008.

\bibitem{Sarkis.M1994}
M.~V. Sarkis.
\newblock {\em Schwarz Preconditioners for Elliptic Problems with Discontinuous
  Coefficients Using Conforming and Non-Conforming Elements}.
\newblock PhD thesis, Courant Institute of Mathematical Science of New York
  University, 1994.

\bibitem{Toselli.A;Widlund.O2005}
A.~Toselli and O.~Widlund.
\newblock {\em Domain Decomposition Methods: Algorithms and Theory}.
\newblock Springer Series in Computational Mathematics, 2005.

\bibitem{Vassilevski.P;Wang.J1995}
P.~S. Vassilevski and J.~Wang.
\newblock An application of the abstract multilevel theory to nonconforming
  finite element methods.
\newblock {\em SIAM Journal on Numerical Analysis}, 32(1):235--248, 1995.

\bibitem{Xu.J1989}
J.~Xu.
\newblock {\em Theory of Multilevel Methods}.
\newblock PhD thesis, Cornell University, 1989.

\bibitem{Xu.J1992a}
J.~Xu.
\newblock Iterative methods by space decomposition and subspace correction.
\newblock {\em SIAM Review}, 34:581--613, 1992.

\bibitem{Xu.J1996}
J.~Xu.
\newblock The auxiliary space method and optimal multigrid preconditioning
  techniques for unstructured meshes.
\newblock {\em Computing}, 56:215--235, 1996.

\bibitem{Xu.J;Zhu.Y2008}
J.~Xu and Y.~Zhu.
\newblock Uniform convergent multigrid methods for elliptic problems with
  strongly discontinuous coefficients.
\newblock {\em Mathematical Models and Methods in Applied Science}, 18(1):77
  --105, 2008.

\bibitem{Xu.J;Zikatanov.L2002}
J.~Xu and L.~Zikatanov.
\newblock The method of alternating projections and the method of subspace
  corrections in {H}ilbert space.
\newblock {\em Journal of The American Mathematical Society}, 15:573--597,
  2002.

\end{thebibliography}

\end{document}